\documentclass[reqno,b5paper]{amsart}

\usepackage{amsmath}
\usepackage{amssymb}
\usepackage{amsthm}
\usepackage{enumerate}
\usepackage[mathscr]{eucal}
\usepackage{eqlist}

\theoremstyle{plain}
\newtheorem{theorem}{Theorem}[section]

\newtheorem{lemma}[theorem]{Lemma}

\theoremstyle{definition}
\newtheorem{definition}[theorem]{Definition}
\newtheorem{example}[theorem]{Example}

\DeclareMathOperator{\Tr}{Tr}
\DeclareMathOperator{\Sp}{Sp}
\newtheorem{note}{Note}
\newtheorem{remark}{Remark}
\numberwithin{equation}{section}
\usepackage{enumerate}

\numberwithin{equation}{section}

\begin{document}
\title[Some integral inequalities for operator AG-convex functions]%
{Some integral inequalities for operator arithmetic-geometrically convex functions}
\author[A. Taghavi, V. Darvish and H. M. Nazari]%
{Ali Taghavi, Vahid Darvish and Haji Mohammad Nazari}

\newcommand{\acr}{\newline\indent}
\address{Department of Mathematics,\\ Faculty of Mathematical
Sciences,\\ University of Mazandaran,\\ P. O. Box 47416-1468,\\
Babolsar, Iran.} \email{taghavi@umz.ac.ir,  vahid.darvish@mail.com (vdarvish.wordpress.com), m.nazari@stu.umz.ac.ir} 


\begin{abstract}
In this paper, we introduce the concept of operator arithmetic-geometrically convex functions for positive linear operators and prove some Hermite-Hadamard type inequalities for these functions. As applications, we obtain trace inequalities for operators which give some refinements of previous results. Moreover, some unitarily invariant norm inequalities are established.
\end{abstract}

\maketitle

\section{\textbf{Introduction and preliminaries }}\label{intro}
Let $\mathcal{A}$ be a sub-algebra of $B(H)$ stand for the commutative  $C^{*}$-algebra of  all bounded linear
operators on a complex Hilbert space $H$ with inner
product $\langle \cdot,\cdot\rangle$. An operator $A\in \mathcal{A}$ is positive and write $A\geq0$ if $\langle
Ax,x\rangle\geq0$ for all $x\in H$. Let $\mathcal{A}^{+}$ stand for all strictly positive operators in $\mathcal{A}$.

Let $A$ be a self-adjoint operator in $\mathcal{A}$. The Gelfand map establishes a $\ast$-isometrically isomorphism $\Phi$ between the set $C(\Sp(A))$ of all continuous functions defined on the spectrum of $A$, denoted $\Sp(A)$, and the $C^{*}$-algebra $C^{*}(A)$ generated by $A$ and the identity operator $1_{H}$ on $H$ as follows:

For any $f,g\in C(\Sp(A)))$ and any $\alpha, \beta\in\mathbb{C}$ we have:
\begin{itemize}
\item
$\Phi(\alpha f+\beta g)=\alpha \Phi(f)+\beta \Phi(g);$
\item
$\Phi(fg)=\Phi(f)\Phi(g)$ and $ \Phi(\bar{f})=\Phi(f)^{*};$
\item
$\|\Phi(f)\|=\|f\|:=\sup_{t\in \Sp(A)}|f(t)|;$
\item
$\Phi(f_{0})=1_{H}$ and $\Phi(f_{1})=A,$ where $f_{0}(t)=1$ and $f_{1}(t)=t$, for $t\in \Sp(A)$.
\end{itemize}
with this notation we define
$$f(A)=\Phi(f) \ \text{for all} \ \ f\in C(\Sp(A))$$
 and we call it the continuous functional calculus for a self-adjoint operator $A$.
 
 If $A$ is a self-adjoint operator and $f$ is a real valued continuous function on $\Sp(A)$, then $f(t)\geq0$ for any $t\in \Sp(A)$ implies that $f(A)\geq0$, i.e. $f(A)$ is a positive operator on $H$. Moreover, if both $f$ and $g$ are real valued functions on $\Sp(A)$ then the following important property holds:
\begin{equation}\label{e3}
f(t)\geq g(t) \ \ \text{for any} \ \ t\in \Sp(A)\ \ \text{implies that} \ \ f(A)\geq g(A),
\end{equation}
in the operator order of $B(H)$, see \cite{zhu}.\\

Let $I$ be an interval in $\mathbb{R}$. Then $f:I\to\mathbb{R}$ is said to be convex (concave) function if 
$$f(\lambda a+(1-\lambda)b)\leq (\geq) \lambda f(a)+(1-\lambda) f(b)$$
for $a,b\in I$ and $\lambda\in [0,1]$.

The following Hermite-Hadamard inequality holds for any convex function $f$ defined on $\mathbb{R}$
\begin{equation}\label{ro}
(b-a)f\left(\frac{a+b}{2}\right)\leq \int_{a}^{b} f(x)dx\leq (b-a)\frac{f(a)+f(b)}{2}, \ \ a,b\in\mathbb{R}.
\end{equation}
It was firstly discovered by Hermite in 1881 in the journal Mathesis (see \cite{mit}). But this result was nowhere mentioned in the mathematical literature and was not widely known as Hermite's result \cite{pec}.

Beckenbach, a leading expert on the history and the theory of convex functions, wrote that this inequality was proven by
Hadamard in 1893 \cite{bec}. In 1974, Mitrinovi\v{c} found Hermite’s note in Mathesis \cite{mit}. Since (\ref{ro}) was known as Hadamard’s
inequality, the inequality is now commonly referred as the Hermite-Hadamard inequality \cite{pec}.

Throughout this paper, we prove some Hermite-Hadamard type inequalities for arithmetic-geometrically convex functions. Moreover, we introduce the concept of operator arithmetic-geometrically convex functions and prove some Hermite-Hadamard type inequalities for these class of functions. These results lead us to obtain some inequalities for trace functional of operators. Finally, we obtain some unitarily invariant norm inequalities for operators.
 \section{\textbf{Inequalities for arithmetic-geometrically convex functions}}
In this section, we prove some new inequalities for arithmetic-geometrically convex functions.

\begin{definition}
A continuous function $f:I\subset\mathbb{R}\to\mathbb{R}^{+}$ is said to be an AG-convex function (arithmetic-geometrically or $\log$ convex function) if 
\begin{equation}\label{ag}
f(\lambda a+(1-\lambda)b)\leq f(a)^{\lambda}f(b)^{1-\lambda}.
\end{equation}
for $a,b\in I$ and $\lambda\in[0,1]$, i.e., $\log f$ is convex.
\end{definition}
\begin{remark}\label{rtab}
The condition (\ref{ag}) can be written as
 \begin{equation}\label{tab}
 f\circ\log (\exp(\lambda a)\exp((1-\lambda)b))\leq (f\circ \log(\exp(a)))^{\lambda}(f\circ\log(\exp(b)))^{1-\lambda},
  \end{equation}
 then we observe that $f:I\to\mathbb{R}^{+}$ is AG-convex on $I$ if and only if $f\circ\log$ is GG-convex (geometrically or multiplicative convex) on $\exp(I):=\{\exp(x), x\in I\}$, by GG-convex we mean a continuous function $f:I\subset \mathbb{R}^{+}\to\mathbb{R}^{+}$ which satisfies in the following 
$$f(a^{\lambda}b^{1-\lambda})\leq f(a)^{\lambda}f(b)^{1-\lambda}$$
for $a,b\in I$ and $\lambda\in[0,1]$ (see \cite{nic}).\\
Note if $I=[a,b]$ then $\exp(I)=[\exp(a),\exp(b)]$.
 \end{remark}

The author of  \cite[p. 158]{nic} showed that every polynomial $P(x)$ with non-negative coefficients is a GG-convex function on $[0,\infty)$. More generally, every real analytic function $f(x)=\sum_{n=0}^{\infty}c_{n}x^{n}$ with non-negative coefficients is a GG-convex function on $(0,R)$ where $R$ denotes the radius of convergence.  This gives some different examples of GG-convex function. It is easy to show that $\exp(x)$ is a GG-convex function.
\begin{remark}
It is well-known that for positive numbers $a$ and $b$
$$\min\{a,b\}\leq G(a,b)=\sqrt{ab}\leq L(a,b)=\frac{b-a}{\ln b-\ln a}\leq A(a,b)=\frac{a+b}{2}\leq \max\{a,b\}.$$
\end{remark}

We proved in \cite[Theorem 3]{taghavi2}  the following inequalities for GG-convex functions

\begin{eqnarray}\label{mes}
f(\sqrt{ab})&\leq& \sqrt{\left(f(a^{\frac{3}{4}}b^{\frac{1}{4}})f(a^{\frac{1}{4}}b^{\frac{3}{4}})\right)}\nonumber\\
&\leq& \exp\left(\frac{1}{\log b-\log a}\int_{a}^{b}\frac{\log\circ f(t)}{t}dt\right)\nonumber\\
&\leq&\sqrt{f(\sqrt{ab})}.\sqrt[4]{f(a)}.\sqrt[4]{f(b)}\nonumber\\
&\leq& \sqrt{f(a)f(b)}.
\end{eqnarray}

\begin{theorem}
Let $f$ be an AG-convex function defined on $[a,b]$. Then, we have
\begin{eqnarray}\label{recall}
f\left(\frac{a+b}{2}\right)&\leq& \sqrt{f\left(\frac{3a+b}{4}\right)f\left(\frac{a+3b}{4}\right)}\nonumber\\
&\leq& \exp\left(\frac{1}{b-a}\int_{a}^{b}\log\circ f(u)du\right)\nonumber\\
&\leq&\sqrt{f\left(\frac{a+b}{2}\right)}.\sqrt[4]{f(a)}.\sqrt[4]{f(b)}\nonumber\\
&\leq& \sqrt{f(a)f(b)},
\end{eqnarray}
where $u=\log t$.
\end{theorem}
\begin{proof}
Let $f:[a,b]\to \mathbb{R}^{+}$ be an AG-convex function, then by  Remark \ref{rtab},  $f\circ\log$ is GG-convex on $[\exp(a),\exp(b)]$.
Applying (\ref{mes}), we obtain
\begin{small}
\begin{eqnarray*}
f\circ\log(\sqrt{\exp(a)\exp(b)})&\leq& \sqrt{f\circ\log\left(\exp(a)^{\frac{3}{4}}\exp(b)^{\frac{1}{4}}\right)f\circ\log\left(\exp(a)^{\frac{1}{4}}\exp(b)^{\frac{3}{4}}\right)}\\
&\leq& \exp\left(\frac{1}{\log\exp(b)-\log\exp(a)}\int_{\exp(a)}^{\exp(b)}\frac{\log\circ f\circ\log(t)}{t}dt\right)\\
&\leq&\sqrt{f\circ\log\sqrt{\exp(a)\exp(b)}}.\sqrt[4]{f\circ\log\exp(a)}.\sqrt[4]{f\circ\log\exp(b)}\\
&\leq& \sqrt{f\circ\log\exp(a)f\circ\log\exp(b)}.
\end{eqnarray*}
\end{small}
Hence, we can write
\begin{small}
\begin{eqnarray*}
f\circ\log\left(\exp\left(\frac{a+b}{2}\right)\right)&\leq& \sqrt{\left(f\circ\log\exp\left(\frac{3a+b}{4}\right)\right)\left(f\circ\log\exp\left(\frac{a+3b}{4}\right)\right)}\\
&\leq& \exp\left(\frac{1}{\log\exp(b)-\log\exp(a)}\int_{\exp(a)}^{\exp(b)}\frac{\log\circ f\circ\log(t)}{t}dt\right)\\
&\leq&\sqrt{f\circ\log\exp\left(\frac{a+b}{2}\right)}.\sqrt[4]{f\circ\log\exp(a)}.\sqrt[4]{f\circ\log\exp(b)}\\
&\leq& \sqrt{f\circ\log\exp(a)f\circ\log\exp(b)}.
\end{eqnarray*}
\end{small}
It follows that

\begin{eqnarray*}
f\left(\frac{a+b}{2}\right)&\leq& \sqrt{f\left(\frac{3a+b}{4}\right)f\left(\frac{a+3b}{4}\right)}\\
&\leq& \exp\left(\frac{1}{b-a}\int_{a}^{b}\log\circ f(u)du\right)\\
&\leq&\sqrt{f\left(\frac{a+b}{2}\right)}.\sqrt[4]{f(a)}.\sqrt[4]{f(b)}\\
&\leq& \sqrt{f(a)f(b)},
\end{eqnarray*}
where $u=\log t$.
\end{proof}

\section{\textbf{Inequalities for operator arithmetic-geometrically convex functions}}\label{gg}

In this section, we prove some Hermite-Hadamard type inequalities for operator arithmetic-geometrically convex function. \\

Here we recall some preliminaries of operator convexity.
\begin{definition}
A real valued continuous function $f$ on an interval $I$ is said to be operator convex if 
$$f(\lambda A+(1-\lambda)B)\leq \lambda f(A)+(1-\lambda)f(B),$$
in the operator order, for all $\lambda\in[0,1]$ and self-adjoint operators $A$ and $B$ in $B(H)$ whose spectra are contained in $I$.
\end{definition}
In \cite{dra3} Dragomir investigated the operator version of the Hermite-Hadamard inequality for
operator convex functions. Let $f:I\to\mathbb{R}$ be an operator convex function on the interval $I$ then, for any self-adjoint operators $A$ and $B$ with spectra in $I$, the following inequalities holds
\begin{eqnarray*}
f\left(\frac{A+B}{2}\right)&\leq&2\int_{\frac{1}{4}}^{\frac{3}{4}} f(tA+(1-t)B)dt\\
&\leq& \frac{1}{2}\left[f\left(\frac{3A+B}{4}\right)+f\left(\frac{A+3B}{4}\right)\right]\\
&\leq& \int_{0}^{1}f\left((1-t)A+tB\right)dt\nonumber\\
&\leq&\frac{1}{2}\left[f\left(\frac{A+B}{2}\right)+\frac{f(A)+f(B)}{2}\right]\\
&\leq& \frac{f(A)+f(B)}{2},
\end{eqnarray*} 
for the first inequality in above, see \cite{taghavi}.

In \cite{taghavi2}, we presented the concept of an operator GG-convex function as follows:
\begin{definition}
A continuous function $f:I\subseteq \mathbb{R}^{+}\to\mathbb{R}^{+}$ is said to be operator GG-convex (geometrically convex)  if 
$$f(A^{\lambda}B^{1-\lambda})\leq f(A)^{\lambda}f(B)^{1-\lambda}$$
for $\lambda\in[0,1]$ and  $A, B\in \mathcal{A}^{+}$ such that $\Sp (A)$, $\Sp(B)\subseteq I$.
\end{definition}

Also, we obtained the following Hermite-Hadamard type inequality for the operator GG-convex function
\begin{equation*}
\log f(\sqrt{AB})\leq \int_{0}^{1}\log f(A^{t}B^{1-t})dt\leq \log \sqrt{f(A)f(B)}
\end{equation*}
for $0\leq t\leq 1$ and $A,B\in \mathcal{A}^{+}$ such that $\Sp (A), \Sp(B)\subseteq I$.

\begin{definition}
A continuous function $f:I\subseteq \mathbb{R}\to\mathbb{R}^{+}$ is said to be  operator AG-convex (concave) if 
$$f(\lambda A+(1-\lambda) B)\leq \  (\geq)  \ f(A)^{\lambda}f(B)^{1-\lambda}$$
for $0\leq\lambda\leq1$ and self-adjoint operators $A$ and $B$ in $B(H)$ whose spectra are contained in $I$.
\end{definition}
\begin{example}\cite[Corollary 7.6.8]{hor}
Let $A$ and $B$ be to positive definite $n\times n$ complex matrices. For $0<\alpha<1$, we have
\begin{equation}\label{det}
|\alpha A+(1-\alpha)B|\geq |A|^{\alpha}|B|^{1-\alpha}
\end{equation}
where $|\cdot|$ denotes determinant of a matrix.
\end{example}
Let $f$ be an operator AG-convex function, for $A$ and $B$ in $\mathcal{A}^{+}$ whose spectra are contained in $I$, then we have
\begin{eqnarray*}
f\left(\frac{A+B}{2}\right)&=& f\left(\frac{\alpha A+(1-\alpha)B}{2}+\frac{(1-\alpha)A+\alpha B}{2}\right)\\
&\leq&\sqrt{f(\alpha A+(1-\alpha)B) f((1-\alpha)A+\alpha B)}\\
&\leq&\sqrt{f(A)^{\alpha}f(B)^{1-\alpha} f(A)^{1-\alpha}f(B)^{\alpha}}\\
&=&\sqrt{f(A)f(B)}.
\end{eqnarray*}
Therefore,
$$f\left(\frac{A+B}{2}\right)\leq\sqrt{f(\alpha A+(1-\alpha)B) f((1-\alpha)A+\alpha B)}\leq \sqrt{f(A)f(B)}.$$
Integrate the above inequality over $[0,1]$, we can write the following
\begin{eqnarray*}
\int_{0}^{1}f\left(\frac{A+B}{2}\right)d\alpha &\leq& \int_{0}^{1}\sqrt{f(\alpha A+(1-\alpha)B) f((1-\alpha)A+\alpha B)}d\alpha\\
&\leq&\int_{0}^{1}\sqrt{f(A)f(B)}d\alpha.
\end{eqnarray*}
So, 
\begin{eqnarray}
f\left(\frac{A+B}{2}\right) &\leq& \int_{0}^{1}\sqrt{f(\alpha A+(1-\alpha)B) f((1-\alpha)A+\alpha B)}d\alpha\nonumber\\
&\leq&\sqrt{f(A)f(B)},\label{jv1}
\end{eqnarray}
for $A$ and $B$ in $\mathcal{A}^{+}$.

\begin{note} We know that $f(t)=t^{\alpha}$ is a concave function for $\alpha\in[0,1]$. Then for positive operator $T\in B(H)$
\begin{eqnarray*}
\|T^{\alpha}x\|&=&\langle T^{\alpha}x,T^{\alpha}x\rangle^{\frac{1}{2}}\\
&=&\langle T^{2\alpha}x,x\rangle^{\frac{1}{2}}\\
&\leq&\langle T^{2}x,x\rangle^{\frac{\alpha}{2}} \ \ \ \ \text{by \cite[Problem IX.8.14]{bha} for concave function}\\
&=&\|Tx\|^{\alpha}.
\end{eqnarray*} 
Taking $\sup$ of the above inequality for $x\in H$ such that $\|x\|=1$, we obtain
\begin{equation}\label{conc}
\|T^{\alpha}\|\leq\|T\|^{\alpha}.
\end{equation}
\end{note}
In order to establish Hermite-Hadamard type inequality for operator AG-convex function, we need the following lemmas. 
\begin{lemma}\cite[Lemma 3]{nag}\label{com}
Let $A$ and $B$ be two operators in $\mathcal{A}^{+}$, and $f$ a continuous function on $\Sp(A)$. Then, $AB=BA$ implies that $f(A)B=Bf(A)$.
\end{lemma}
\noindent Since $f(t)=t^{\lambda}$ is a continuous function for $\lambda\in [0,1]$  and $\mathcal{A}$ is a commutative $C^{*}$-algebra, we have
$A^{\lambda}B=BA^{\lambda}$. Moreover, by applying the above lemma for $f(t)=t^{1-\lambda}$ again, we have
$A^{\lambda}B^{1-\lambda}=B^{1-\lambda}A^{\lambda}$, for  operators $A$ and $B$ in $\mathcal{A}^{+}$. It means that $A^{\lambda}$ and $B^{1-\lambda}$ commute together whenever $A$ and $B$ commute.
\begin{lemma}\label{gconvex}\cite[Lemma 3]{taghavi2}
Let $A$ and $B$ be two operators in $\mathcal{A}^{+}$. Then $$\{A^{\lambda}B^{1-\lambda} : 0\leq \lambda\leq 1\}$$ is convex.
\end{lemma}
\begin{lemma}\cite[Theorem 5.3]{zhu}
Let $A$ and $B$ be in a Banach algebra such that $AB=BA$. Then 
$$\Sp(AB)\subset\Sp (A)\Sp (B).$$
 \end{lemma}
 Let $A$ and $B$ be two positive operators in $\mathcal{A}$ with spectra in $I$. Now, Lemma \ref{com} and functional calculus \cite[Theorem 10.3 (c)]{zhu} imply that
 \begin{equation*}
 \Sp(A^{\lambda}B^{1-\lambda})\subset \Sp(A^{\lambda})\Sp(B^{1-\lambda})=\Sp(A)^{\lambda}\Sp(B)^{1-\lambda}\subseteq I
 \end{equation*}
 for $0\leq\lambda\leq1$.

\begin{lemma}\label{lemn}
Let $f$ be an operator GG-convex function. Then
$\varphi_{A,B}(t)=\| f(A^{t}B^{1-t})\|$ is an AG-convex function for $0\leq t\leq 1$ and $A,B\in \mathcal{A}^{+}$ such that $\Sp (A), \Sp(B)\subseteq I$.
\end{lemma}
\begin{proof}
We should prove that $$\varphi_{A,B}(\alpha u+(1-\alpha)v)\leq \varphi_{A,B}(u)^{\alpha}\varphi_{A,B}(v)^{1-\alpha}$$ for $\alpha\in [0,1]$ and $u,v\in [0,1]$.
\begin{eqnarray*}
\varphi_{A,B}(\alpha u+(1-\alpha)v)&=&\| f(A^{\alpha u+(1-\alpha)v}B^{1-(\alpha u+(1-\alpha)v)})\|\\
&=& \| f(A^{\alpha u+(1-\alpha)v}B^{\alpha(1- u)+(1-\alpha)(1-v)})\|\\
&=&\| f(A^{\alpha u}B^{\alpha(1- u)}A^{(1-\alpha)v}B^{(1-\alpha)(1-v)})\|\\
&=&\| f((A^{u}B^{1- u})^{\alpha}(A^{v}B^{1-v})^{1-\alpha})\|\\
&\leq&\| f(A^{u}B^{1- u})^{\alpha}f(A^{v}B^{1-v})^{1-\alpha}\| \ \ \ \ \ \ \text{GG-convexity of $f$}\\
&\leq&\| f(A^{u}B^{1- u})^{\alpha}\|\| f(A^{v}B^{1-v})^{1-\alpha}\| \\
&\leq&\| f(A^{u}B^{1- u})\|^{\alpha}\| f(A^{v}B^{1-v})\|^{1-\alpha} \ \ \ \ \  \text{by (\ref{conc})} \\
&=&\varphi_{A,B}(u)^{\alpha}\varphi_{A,B}(v)^{1-\alpha},
\end{eqnarray*}
for $A,B\in \mathcal{A}^{+}$ such that $\Sp (A), \Sp(B)\subseteq I$.
\end{proof}

Here, we recall inequalities (\ref{recall}) for AG-convex function $\varphi$
\begin{eqnarray}\label{var}
\varphi\left(\frac{a+b}{2}\right)&\leq& \sqrt{\varphi\left(\frac{3a+b}{4}\right)\varphi\left(\frac{a+3b}{4}\right)}\nonumber\\
&\leq& \exp\left(\frac{1}{b-a}\int_{a}^{b}\log\circ \varphi(u)du\right)\nonumber\\
&\leq&\sqrt{\varphi\left(\frac{a+b}{2}\right)}.\sqrt[4]{\varphi(a)}.\sqrt[4]{\varphi(b)}\nonumber\\
&\leq& \sqrt{\varphi(a)\varphi(b)}.
\end{eqnarray}
\begin{theorem}
Let $f$ be an operator GG-convex function.
Then
\begin{eqnarray}\label{mainag}
\|f(\sqrt{AB})\|&\leq& \sqrt{\|f(A^{\frac{1}{4}}B^{\frac{3}{4}})x\|f(A^{\frac{3}{4}}B^{\frac{1}{4}})\|}\nonumber\\
&\leq& \exp\left(\int_{0}^{1}\log \|f(A^{u}B^{1-u})\|du\right)\nonumber\\
&\leq&\sqrt{\|f(A^{\frac{1}{2}}B^{\frac{1}{2}})\|}.\sqrt[4]{\|f(B)\|}.\sqrt[4]{\|f(A)\|}\nonumber\\
&\leq& \sqrt{\|f(A)\|\|f(B)\|},
\end{eqnarray}
for $0\leq u\leq 1$ and $A,B\in \mathcal{A}^{+}$ such that $\Sp (A), \Sp(B)\subseteq I$.
\end{theorem}
\begin{proof}
Since $f$ is an operator GG-convex function,  by Lemma \ref{lemn} we can say $\varphi_{A,B}(t)=\|f(A^{t}B^{1-t})\|$ is an AG-convex function on $[0,1]$ .
Applying inequalities (\ref{var}) for $a=0$ and $b=1$, we have
\begin{eqnarray*}
\varphi\left(\frac{1}{2}\right)&\leq& \sqrt{\varphi\left(\frac{1}{4}\right)\varphi\left(\frac{3}{4}\right)}\\
&\leq& \exp\left(\int_{0}^{1}\log\circ \varphi(u)du\right)\\
&\leq&\sqrt{\varphi\left(\frac{1}{2}\right)}.\sqrt[4]{\varphi(0)}.\sqrt[4]{\varphi(1)}\\
&\leq& \sqrt{\varphi(0)\varphi(1)}.
\end{eqnarray*}
It follows that
\begin{eqnarray*}
\|f(A^{\frac{1}{2}}B^{\frac{1}{2}})\|&\leq& \sqrt{\|f(A^{\frac{1}{4}}B^{\frac{3}{4}})\|f(A^{\frac{3}{4}}B^{\frac{1}{4}})\|}\\
&\leq& \exp\left(\int_{0}^{1}\log \|f(A^{u}B^{1-u})\|du\right)\\
&\leq&\sqrt{\|f(A^{\frac{1}{2}}B^{\frac{1}{2}})\|}.\sqrt[4]{\|f(B)\|}.\sqrt[4]{\|f(A)\|}\\
&\leq& \sqrt{\|f(B)\|\|f(A)\|}.
\end{eqnarray*}
\end{proof}

\begin{remark}
Let $A, B\in \mathcal{A}$ and $A\leq B$, by continuous functional calculus \cite[Theorem 10.3 (b)]{zhu}, we can easily obtain $\exp(A)\leq \exp(B)$. This means $\exp(t)$ is operator monotone on $[0,\infty)$ for $A,B\in\mathcal{A}$.

 On the other hand, like the classical case, the arithmetic-geometric mean inequality
holds for operators as following
\begin{equation}\label{ag2}
A^{\frac{1}{2}}\left(A^{-\frac{1}{2}}BA^{-\frac{1}{2}}\right)^{\nu}A^{\frac{1}{2}}\leq (1-\nu) A+\nu B, \ \ \nu\in[0,1]
\end{equation}
with respect to operator order for positive non-commutative operators in $B(H)$. Whenever, $A$ and $B$ commute together, then inequality (\ref{ag2}) reduces to 
\begin{equation*}
A^{1-\nu}B^{\nu}\leq (1-\nu)A+\nu B, \ \ \nu\in[0,1].
\end{equation*}
Since $\exp(t)$ is an operator monotone function, by the above inequality we have 
\begin{eqnarray*}
\exp\left({A^{1-\nu}B^{\nu}}\right)&\leq & \exp\left((1-\nu)A+\nu B\right)\\
&=& \exp((1-\nu)A)\exp(\nu B)\\
&=& \exp(A)^{1-\nu}\exp(B)^{\nu},
\end{eqnarray*}
for $A, B\in \mathcal{A}^{+}$ and $\nu\in [0,1]$. So, in this case $\exp(t)$ is an operator geometrically convex function on $[0,\infty)$.

By inequalities (\ref{mainag}), we obtain
\begin{eqnarray}\label{rem2}
\|\exp(\sqrt{AB})\|&\leq& \sqrt{\|\exp(A^{\frac{1}{4}}B^{\frac{3}{4}})\|\exp(A^{\frac{3}{4}}B^{\frac{1}{4}})\|}\nonumber\\
&\leq& \exp\left(\int_{0}^{1}\log \|\exp(A^{u}B^{1-u})\|du\right)\nonumber\\
&\leq&\sqrt{\|\exp(A^{\frac{1}{2}}B^{\frac{1}{2}})\|}.\sqrt[4]{\|\exp(B)\|}.\sqrt[4]{\|\exp(A)\|}\nonumber\\
&\leq& \sqrt{\|\exp(A)\|\|\exp(B)\|}.
\end{eqnarray}
\end{remark}
\begin{remark}
Let $\{e_{i}\}_{i\in I}$ be an orthonormal basis of $H$, we say that $A\in B(H)$ is \textit{trace class} if 
\begin{equation}\label{dd7}
\|A\|_{1}:=\sum_{i\in I}\langle |A|e_{i},e_{i}\rangle <\infty.
\end{equation}
The definition of $\|A\|_{1}$ does not depend on the choice of the orthonormal basis $\{e_{i}\}_{i\in I}$. We denote by $B_{1}(H)$ the set of trace class operators in $B(H)$.

We define the \textit{trace} of a trace class operator $A\in B_{1}(H)$ to be 
\begin{equation}\label{dd9}
\Tr(A):=\sum_{i\in I}\langle Ae_{i}, e_{i}\rangle,
\end{equation}
where $\{e_{i}\}_{i\in I}$ an orthonormal basis of $H$. 

Note that this coincides with the usual definition of the trace if $H$ is finite-dimensional. We observe that the series (\ref{dd9}) converges absolutely.\\

The following result collects some properties of the trace:

\begin{theorem}
We have

(i) If $A\in B_{1}(H)$ then $A^{*}\in B_{1}(H)$ and 
\begin{equation}\label{dd10}
\Tr(A^{*})=\overline{\Tr(A)};
\end{equation}

(ii) If $A\in B_{1}(H)$ and $T\in B(H)$, then $AT, TA\in B_{1}(H)$ and
\begin{equation}\label{dd11}
\Tr(AT)=\Tr(TA) \ \ \ and \ \ \ |\Tr(AT)|\leq \|A\|_{1}\|T\|;
\end{equation}

(iii) $\Tr(\cdot)$ is a bounded linear functional on $B_{1}(H)$ with $\|\Tr\|=1$;

(iv) If $A, B\in B_{1}(H)$ then  $\Tr(AB)=\Tr(BA)$.
\end{theorem}
For the theory of trace functionals and their applications the reader is referred to \cite{sim}.

We know that $f(t)=\Tr(t)$ is an operator GG-convex function \cite[p.513]{hor}, i.e. 
$$\Tr(A^{t}B^{1-t})\leq \Tr (A)^{t}\Tr(B)^{1-t}$$
for $0\leq t\leq 1$ and positive operators $A,B\in B_{1}(H) $.\\
For $A,B\geq 0$ we have $\Tr (AB)\leq \Tr(A)\Tr(B)$. Also, since  $f(t)=t^{\frac{1}{2}}$ is monotone  we have
\begin{equation*}
\sqrt{\Tr(AB)}\leq \sqrt{\Tr(A)\Tr(B)}
\end{equation*}
for positive operator $A$ and $B$ in $B_{1}(H)$.\\
For commutative case, we have
\begin{equation}\label{ot1}
\sqrt{\Tr(AB)}\leq \Tr(\sqrt{AB})\leq \sqrt{\Tr(A)\Tr(B)},
\end{equation}
since $(\Tr (AB))^{\frac{1}{2}}\leq \Tr(AB)^{\frac{1}{2}}$.\\
By inequalities (\ref{ot1}) and (\ref{mainag}), we obtain
\begin{eqnarray}
\sqrt{\Tr(AB)}&\leq&\Tr(\sqrt{AB})\nonumber\\
&\leq& \sqrt{\Tr(A^{\frac{1}{4}}B^{\frac{3}{4}})\Tr(A^{\frac{3}{4}}B^{\frac{1}{4}})}\nonumber\\
&\leq& \exp\left(\int_{0}^{1}\log \Tr(A^{u}B^{1-u})du\right)\nonumber\\
&\leq&\sqrt{\Tr(A^{\frac{1}{2}}B^{\frac{1}{2}})}.\sqrt[4]{\Tr(B)}.\sqrt[4]{\Tr(A)}\nonumber\\
&\leq& \sqrt{\Tr(A)\Tr(B)}.
\end{eqnarray}

Let replace $A$ and $B$ by $A^{2}$ and $B^{2}$ in the above inequality, respectively. By applying commutativity of algebra and knowing that $\Tr(A)^{2}\leq (\Tr A)^{2}$ for positive operator $A$, we have
\begin{eqnarray}
\Tr(AB)&\leq& \sqrt{\Tr(A^{\frac{1}{2}}B^{\frac{3}{2}})Tr(A^{\frac{3}{2}}B^{\frac{1}{2}})}\nonumber\\
&\leq& \exp\left(\int_{0}^{1}\log \Tr\left(A^{2u}B^{2(1-u)}\right)du\right)\nonumber\\
&\leq&\sqrt{\Tr(AB)}.\sqrt[4]{\Tr(B^{2})}.\sqrt[4]{\Tr(A^{2})}\nonumber\\
&\leq&\Tr(A)\Tr(B).
\end{eqnarray}
\end{remark}
\section{\textbf{Some unitarily invariant norm inequalities for non-commutative operators}}

In this section we prove some unitarily invariant norm inequalities for operators.

We consider the wide class of unitarily invariant
norms $|||\cdot|||$. Each of these norms is defined on an ideal in
$B(H)$ and it will be implicitly understood that when we talk of
$|||T|||$, then the operator $T$ belongs to the norm ideal
associated with $|||\cdot|||$. Each unitarily invariant norm
$|||\cdot|||$ is characterized by the invariance property
$|||UTV|||=|||T|||$ for all operators $T$ in the norm ideal
associated with $|||\cdot|||$ and for all unitary operators $U$ and
$V$ in $B(H)$.
 For $1\leq
p<\infty$, the Schatten $p$-norm of a compact operator $A$ is
defined by $\|A\|_{p}=(\Tr |A|^{p})^{1/p}$, where $\Tr$ is the usual
trace functional. These norms are
special examples of the more general class of the Schatten
$p$-norms which are unitarily invariant \cite{bha}.

The author of \cite{kit} proved that if $A, B, X\in B(H)$ such that $A, B$ are positive operators, then for $0\leq\nu\leq 1$ we have
\begin{equation}\label{ksab}
|||A^{\nu}XB^{1-\nu}|||\leq |||AX|||^{\nu}|||XB|||^{1-\nu}.
\end{equation}
\\
The idea of the following lemma comes from \cite[Lemma 2.1]{sab}.
\begin{lemma}\label{sabb}
Let $A,B,X\in B(H)$ such that $A$ and $B$ are positive operators. Then
$$f(t)=|||A^{t}XB^{1-t}|||$$
is AG-convex for $t\in [0,1]$.
\end{lemma}
\begin{proof}
We should prove that $f(\alpha t+(1-\alpha)s)\leq f(t)^{\alpha}f(s)^{1-\alpha}$, for $\alpha\in(0,1)$ and $s,t\in[0,1]$.
\begin{eqnarray*}
f(\alpha t+(1-\alpha)s)&=&|||A^{\alpha t+(1-\alpha)s}XB^{1-(\alpha t+(1-\alpha)s)}|||\\
&=&|||A^{\alpha(t-s)} A^{s}XB^{1-t}B^{(1-\alpha)(t-s)}|||\\
&\leq&|||A^{t-s}A^{s}XB^{1-t}|||^{\alpha}|||A^{s}XB^{1-t}B^{t-s}|||^{1-\alpha} \ \ \ \ \ \ \  \ \text{by (\ref{ksab})}\\
&=&|||A^{t}XB^{1-t}|||^{\alpha}|||A^{s}XB^{1-s}|||^{1-\alpha}\\
&=& f(t)^{\alpha}f(s)^{1-\alpha}.
\end{eqnarray*}
So, $f(t)=|||A^{t}XB^{1-t}|||$ is an AG-convex function.
\end{proof}

Applying inequalities (\ref{recall}) to the function $f(t)=|||A^{t}XB^{1-t}|||$ on the interval $[\nu,1-\nu]$ when $\nu\in[0,\frac{1}{2})$ and on the interval $[1-\nu,\nu]$ when $\nu\in (\frac{1}{2},1]$, we obtain the following theorem.
\begin{theorem}\label{sta}
Let $A,B,X\in B(H)$ such that $A$ and $B$ are positive operators. Then
\begin{enumerate}
\item For $\nu\in[0,\frac{1}{2})$, we have
\begin{eqnarray*}
|||A^{\frac{1}{2}}XB^{\frac{1}{2}}|||&\leq& |||A^{\frac{1+2\nu}{4}}XB^{\frac{3-2\nu}{4}}|||^{\frac{1}{2}}|||A^{\frac{3-2\nu}{4}}XB^{\frac{1+2\nu}{4}}|||^{\frac{1}{2}}\\
&\leq& \exp\left(\frac{1}{1-2\nu}\int_{\nu}^{1-\nu}\log |||A^{u}XB^{1-u}|||du\right)\\\
&\leq&|||A^{\frac{1}{2}}XB^{\frac{1}{2}}|||^{\frac{1}{2}}|||A^{\nu}XB^{1-\nu}|||^{\frac{1}{4}}|||A^{1-\nu}XB^{\nu}|||^{\frac{1}{4}}\\
&\leq& |||A^{\nu}XB^{1-\nu}|||^{\frac{1}{2}}|||A^{1-\nu}XB^{\nu}|||^{\frac{1}{2}}.
\end{eqnarray*}
\item
Also, for $\nu\in(\frac{1}{2},1]$, we have
\begin{eqnarray*}
|||A^{\frac{1}{2}}XB^{\frac{1}{2}}|||&\leq& |||A^{\frac{1+2\nu}{4}}XB^{\frac{3-2\nu}{4}}|||^{\frac{1}{2}}|||A^{\frac{3-2\nu}{4}}XB^{\frac{1+2\nu}{4}}|||^{\frac{1}{2}}\\
&\leq& \exp\left(\frac{1}{2\nu-1}\int_{1-\nu}^{\nu}\log |||A^{u}XB^{1-u}|||du\right)\\\
&\leq&|||A^{\frac{1}{2}}XB^{\frac{1}{2}}|||^{\frac{1}{2}}|||A^{\nu}XB^{1-\nu}|||^{\frac{1}{4}}|||A^{1-\nu}XB^{\nu}|||^{\frac{1}{4}}\\
&\leq& |||A^{\nu}XB^{1-\nu}|||^{\frac{1}{2}}|||A^{1-\nu}XB^{\nu}|||^{\frac{1}{2}}.
\end{eqnarray*}
\end{enumerate}
\end{theorem}
\begin{proof}
Let $\nu\in[0,\frac{1}{2})$, then we have by inequalities (\ref{recall}) and Lemma \ref{sabb} that
\begin{eqnarray*}
f\left(\frac{\nu+1-\nu}{2}\right)&\leq& \sqrt{f\left(\frac{3\nu+1-\nu}{4}\right)f\left(\frac{\nu+3(1-\nu)}{4}\right)}\\
&\leq& \exp\left(\frac{1}{1-2\nu}\int_{\nu}^{1-\nu}\log\circ f(u)du\right)\\
&\leq&\sqrt{f\left(\frac{\nu+1-\nu}{2}\right)}.\sqrt[4]{f(\nu)}.\sqrt[4]{f(1-\nu)}\\
&\leq& \sqrt{f(\nu)f(1-\nu)},
\end{eqnarray*}
and so
\begin{eqnarray*}
|||A^{\frac{1}{2}}XB^{\frac{1}{2}}|||&\leq& |||A^{\frac{1+2\nu}{4}}XB^{\frac{3-2\nu}{4}}|||^{\frac{1}{2}}|||A^{\frac{3-2\nu}{4}}XB^{\frac{1+2\nu}{4}}|||^{\frac{1}{2}}\\
&\leq& \exp\left(\frac{1}{1-2\nu}\int_{\nu}^{1-\nu}\log |||A^{u}XB^{1-u}|||du\right)\\
&\leq&|||A^{\frac{1}{2}}XB^{\frac{1}{2}}|||^{\frac{1}{2}}|||A^{\nu}XB^{1-\nu}|||^{\frac{1}{4}}|||A^{1-\nu}XB^{\nu}|||^{\frac{1}{4}}\\
&\leq& |||A^{\nu}XB^{1-\nu}|||^{\frac{1}{2}}|||A^{1-\nu}XB^{\nu}|||^{\frac{1}{2}}.
\end{eqnarray*}
Similarily, for $\nu\in(\frac{1}{2},1]$ we have
\begin{eqnarray*}
|||A^{\frac{1}{2}}XB^{\frac{1}{2}}|||&\leq& |||A^{\frac{1+2\nu}{4}}XB^{\frac{3-2\nu}{4}}|||^{\frac{1}{2}}|||A^{\frac{3-2\nu}{4}}XB^{\frac{1+2\nu}{4}}|||^{\frac{1}{2}}\\
&\leq& \exp\left(\frac{1}{2\nu-1}\int_{1-\nu}^{\nu}\log |||A^{u}XB^{1-u}|||du\right)\\
&\leq&|||A^{\frac{1}{2}}XB^{\frac{1}{2}}|||^{\frac{1}{2}}|||A^{\nu}XB^{1-\nu}|||^{\frac{1}{4}}|||A^{1-\nu}XB^{\nu}|||^{\frac{1}{4}}\\
&\leq& |||A^{\nu}XB^{1-\nu}|||^{\frac{1}{2}}|||A^{1-\nu}XB^{\nu}|||^{\frac{1}{2}}.
\end{eqnarray*}
This completes the proof.
\end{proof}
Applying inequalities (\ref{recall}) to the function $f(t)=|||A^{t}XB^{1-t}|||$ on the interval $[0,\nu]$ when $0<\nu\leq\frac{1}{2}$ and on the interval $[\nu,1]$ when $\frac{1}{2}\leq\nu<1$, we obtain the following theorem.
\begin{theorem}
Let $A,B,X\in B(H)$ such that $A$ and $B$ are positive operators. Then
\begin{enumerate}
\item For $0<\nu\leq\frac{1}{2}$, we have
\begin{eqnarray*}
|||A^{\frac{\nu}{2}}XB^{1-\frac{\nu}{2}}|||&\leq& |||A^{\frac{\nu}{4}}XB^{1-\frac{\nu}{4}}|||^{\frac{1}{2}}|||A^{\frac{3\nu}{4}}XB^{1-\frac{3\nu}{4}}|||^{\frac{1}{2}}\\
&\leq& \exp\left(\frac{1}{\nu}\int_{0}^{\nu}\log |||A^{u}XB^{1-u}|||du\right)\\\
&\leq&|||A^{\frac{\nu}{2}}XB^{1-\frac{\nu}{2}}|||^{\frac{1}{2}}|||XB|||^{\frac{1}{4}}|||A^{\nu}XB^{1-\nu}|||^{\frac{1}{4}}\\
&\leq& |||XB|||^{\frac{1}{2}}|||A^{\nu}XB^{1-\nu}|||^{\frac{1}{2}}.
\end{eqnarray*}
\item
Also, for $\frac{1}{2}\leq\nu\leq1$, we have
\begin{eqnarray*}
|||A^{\frac{\nu+1}{2}}XB^{1-\frac{\nu+1}{2}}|||&\leq& |||A^{\frac{3\nu+1}{4}}XB^{1-\frac{3\nu+1}{4}}|||^{\frac{1}{2}}|||A^{\frac{\nu+3}{4}}XB^{1-\frac{\nu+3}{4}}|||^{\frac{1}{2}}\\
&\leq& \exp\left(\frac{1}{1-\nu}\int_{\nu}^{1}\log |||A^{u}XB^{1-u}|||du\right)\\\
&\leq&|||A^{\frac{\nu+1}{2}}XB^{1-\frac{\nu+1}{2}}|||^{\frac{1}{2}}|||A^{\nu}XB^{1-\nu}|||^{\frac{1}{4}}|||AX|||^{\frac{1}{4}}\\
&\leq& |||A^{\nu}XB^{1-\nu}|||^{\frac{1}{2}}|||AX|||^{\frac{1}{2}}.
\end{eqnarray*}
\end{enumerate}
\end{theorem}
\begin{proof}
It is similar to the proof of Theorem \ref{sta}.
\end{proof}
\begin{theorem}
Let $A,B,X\in B(H)$ such that $A$ and $B$ are positive operators. Then
\begin{eqnarray}
|||A^{\frac{1}{2}}XB^{\frac{1}{2}}|||&\leq& |||A^{\frac{1}{4}}XB^{\frac{3}{4}}|||^{\frac{1}{2}}|||A^{\frac{3}{4}}XB^{\frac{1}{4}}|||^{\frac{1}{2}}\nonumber\\
&\leq& \exp\left(\int_{0}^{1}\log |||A^{u}XB^{1-u}|||du\right)\nonumber\\
&\leq&|||A^{\frac{1}{2}}XB^{\frac{1}{2}}|||^{\frac{1}{2}}|||AX|||^{\frac{1}{4}}|||XB|||^{\frac{1}{4}}\nonumber\\
&\leq& |||AX|||^{\frac{1}{2}}|||XB|||^{\frac{1}{2}}. \label{thef}
\end{eqnarray}
\end{theorem}
\begin{proof}
It is easy to verify by inequalities (\ref{recall}) and Lemma \ref{sabb} for $a=0$ and $b=1$.
\end{proof}
Note that our inequalities (\ref{thef}) give a refinement of inequality (\ref{ksab}) when $\nu=\frac{1}{2}$.
\begin{lemma}\label{sabb2}
Let $A,B,X\in B(H)$ such that $A$ and $B$ are positive operators. Then
$$f(s)=|||A^{s}XB^{s}|||$$
is AG-convex for $s\in [0,1]$.
\end{lemma}
\begin{proof}
Similar to Lemma \ref{sabb}, we should prove that $f(\alpha t+(1-\alpha)s)\leq f(t)^{\alpha}f(s)^{1-\alpha}$, for $\alpha\in(0,1)$ and $s,t\in[0,1]$.
\begin{eqnarray*}
f(\alpha t+(1-\alpha)s)&=&|||A^{\alpha t+(1-\alpha)s}XB^{\alpha t+(1-\alpha)s}|||\\
&=&|||A^{\alpha(t-s)} A^{s}XB^{t}B^{(1-\alpha)(s-t)}|||\\
&\leq&|||A^{t-s}A^{s}XB^{t}|||^{\alpha}|||A^{s}XB^{t}B^{s-t}|||^{1-\alpha} \ \ \ \ \ \ \  \ \text{by (\ref{ksab})}\\
&=&|||A^{t}XB^{t}|||^{\alpha}|||A^{s}XB^{s}|||^{1-\alpha}\\
&=& f(t)^{\alpha}f(s)^{1-\alpha}.
\end{eqnarray*}
So, $f(s)=|||A^{s}XB^{s}|||$ is an AG-convex function.
\end{proof}
\begin{theorem}
Let $A,B,X\in B(H)$ such that $A$ and $B$ are positive operators. Then
\begin{eqnarray}
|||A^{\frac{1}{2}}XB^{\frac{1}{2}}|||&\leq& |||A^{\frac{1}{4}}XB^{\frac{1}{4}}|||^{\frac{1}{2}}|||A^{\frac{3}{4}}XB^{\frac{3}{4}}|||^{\frac{1}{2}}\nonumber\\
&\leq& \exp\left(\int_{0}^{1}\log |||A^{u}XB^{u}|||du\right)\nonumber\\
&\leq&|||A^{\frac{1}{2}}XB^{\frac{1}{2}}|||^{\frac{1}{2}}|||X|||^{\frac{1}{4}}|||AXB|||^{\frac{1}{4}}\nonumber\\
&\leq& |||X|||^{\frac{1}{2}}|||AXB|||^{\frac{1}{2}}. \label{thef2}
\end{eqnarray}
\end{theorem}
\begin{proof}
Apply inequalities (\ref{recall}) and Lemma \ref{sabb2} for $a=0$ and $b=1$.
\end{proof}
\begin{remark}
The above inequalities give a refinement of inequality \cite[Theorem 1]{kit} when $\nu=\frac{1}{2}$ which states that if $A,B,X\in B(H)$ such that $A,B$ are positive operators then
$$|||A^{\nu}XB^{\nu}|||\leq |||X|||^{1-\nu}|||AXB|||^{\nu}$$
for $0\leq\nu\leq1$.
\end{remark}

\end{document}